\documentclass[final, hidelinks]{siamart220329}

\usepackage[utf8]{inputenc}
\usepackage{amsmath, amssymb, bm, mathrsfs}
\usepackage{csquotes, cases, enumitem}
\usepackage{graphicx, float, subcaption, adjustbox}
\usepackage{algpseudocode}
\usepackage{array, tabularx, colortbl}
\usepackage{color, xcolor}
\usepackage[para]{footmisc}
\usepackage[normalem]{ulem}
\counterwithin{figure}{section}
\counterwithin{table}{section}
\numberwithin{equation}{section}
\graphicspath{{Figures/}}
\definecolor{darkred}{rgb}{0.8,0,0}

\DeclareMathOperator{\trace}{\operatorname{trace}}
\DeclareMathOperator{\vectorization}{\operatorname{vec}}

\DeclareMathOperator{\diag}{\operatorname{diag}}

\DeclareMathOperator{\sparsity}{\operatorname{sp}}

\newenvironment{frcseries}{\fontfamily{frc}\selectfont}{}

\newsiamremark{remark}{Remark}

\title{Preconditioning techniques for generalized Sylvester matrix equations}
\author{Yannis Voet\thanks{MNS, Institute of Mathematics, École polytechnique fédérale de Lausanne, Station 8, CH-1015 Lausanne, Switzerland (\email{yannis.voet@epfl.ch})}}

\begin{document}
\maketitle

\begin{abstract}
Sylvester matrix equations are ubiquitous in scientific computing. However, few solution techniques exist for their generalized multiterm version, as they now arise in an increasingly large number of applications. In this work, we consider algebraic parameter-free preconditioning techniques for the iterative solution of generalized multiterm Sylvester equations. They consist in constructing low Kronecker rank approximations of either the operator itself or its inverse. While the former requires solving standard Sylvester equations in each iteration, the latter only requires matrix-matrix multiplications, which are highly optimized on modern computer architectures. Moreover, low Kronecker rank approximate inverses can be easily combined with sparse approximate inverse techniques, thereby enhancing their performance with little or no damage to their effectiveness.
\end{abstract}

\begin{keywords}
Multiterm Sylvester equations, Low Kronecker rank, Nearest Kronecker product, Alternating least squares, Sparse approximate inverse
\end{keywords}
\begin{MSCcodes}
65F08, 65F45, 65F50
\end{MSCcodes}

\section{Introduction}
\label{se: introduction}
We consider the numerical solution of generalized multiterm Sylvester matrix equations
\begin{equation}
\label{eq: generalized_sylvester}
    \sum_{k=1}^r B_kXA_k^T=E
\end{equation}
where $A_k \in \mathbb{R}^{n \times n}$, $B_k \in \mathbb{R}^{m \times m}$ for all $k=1,\dots,r$ and $X,E \in \mathbb{R}^{m \times n}$. 
Generalized Sylvester equations are at the forefront of many applications in scientific computing. The single-term equation ($r=1$) is its simplest instance and appears for tensorized finite element discretizations of certain time dependent partial differential equations (PDEs) \cite{gao2013kronecker, gao2014fast, loli2021easy}. If $A_1$ and $B_1$ are invertible, the solution is particularly simple since $X=B_1^{-1}EA_1^{-T}$ only requires solving $n$ linear systems with $B_1$ and $m$ linear systems with $A_1$. If $A_1=I_n$, the equation reduces to a standard linear system with multiple right-hand sides.

The two-term equation ($r=2$) includes in particular the (standard) Sylvester, Lyapunov and Stein equations. They appear in various applications, including block-diagonalization of block triangular matrices \cite{higham2002accuracy, stewart1973error}, finite difference \cite{grasedyck2004existence, kressner2010krylov} and finite element \cite{gao2015preconditioners, sangalli2016isogeometric} discretizations of certain PDEs, eigenvalue problems \cite{stewart1973error, stewart1990matrix, chu1987exclusion}, stability and control theory \cite{gajic2008lyapunov} and the stage equations of implicit Runge-Kutta methods \cite{gander2023spectral}. Sylvester equations are also the main building block for iteratively solving more complicated multiterm \cite{damm2008direct, benner2013low, palitta2016matrix} or nonlinear matrix equations, such as the Riccati equation \cite{stewart1973error, benner2013numerical}.

While originally considered of theoretical interest, the generalized multiterm Sylvester equation with $r>2$ now arises in an increasingly large number of applications including finite difference \cite{palitta2016matrix} and finite element \cite{ernst2009efficient, ullmann2010kronecker, mantzaflaris2017low, scholz2018partial} discretizations of (stochastic) PDEs, model reduction \cite{benner2013low, damm2008direct, shank2016efficient}, eigenvalue problems \cite{ringh2018sylvester} and computational neuroscience \cite{kurschner2019greedy}. In model reduction of control systems, the equation often takes a special form consisting of a standard Lyapunov part and additional positive low-rank terms \cite{benner2013low, benner2013numerical, damm2008direct, shank2016efficient}, e.g.
\begin{equation}
\label{eq: Lyapunov_plus_positive}
     AX + XA^T + \sum_{k=1}^r N_kXN_k^T=E.
\end{equation}
Such equations are sometimes referred to as \emph{Lyapunov-plus-positive} equations \cite{benner2013numerical}. In tensorized finite element discretizations (e.g. isogeometric analysis and the spectral element method), the equation arises in its full generality and typically features sparse banded coefficient matrices \cite{mantzaflaris2017low, scholz2018partial, hofreither2018black}. In a similar spirit, discretizations of integro-differential equations may lead to \cref{eq: generalized_sylvester} with both sparse and dense (low-rank) coefficient matrices \cite{massei2018solving}. The growing number of applications is driving interest for solution techniques capable of handling \cref{eq: generalized_sylvester} in its full generality. It is well-known that solving \cref{eq: generalized_sylvester} is equivalent to solving the linear system (see e.g. \cite[Lemma 4.3.1]{horn1991topics} for a derivation and \cite{lancaster1970explicit} for an early mention)
\begin{equation}
\label{eq: kronecker_format}
    \left(\sum_{k=1}^r A_k \otimes B_k\right)\mathbf{x}=\mathbf{e}
\end{equation}
with $\mathbf{x}=\vectorization(X)$ and $\mathbf{e}=\vectorization(E)$, where the vectorization of a matrix $A$, denoted $\vectorization(A)$, stacks the columns of $A$ on top of each other. The coefficient matrix in \cref{eq: kronecker_format} is the sum of $r$ Kronecker products. Such a matrix is said to have Kronecker rank $r$ if $r$ is the smallest number of terms in the sum \cite[Definition 4]{grasedyck2004existence}. However, due to the humongous size of this matrix, solution strategies preferably avoid the Kronecker formulation, unless excellent preconditioners are available.

Strategies for solving \cref{eq: generalized_sylvester} instead most critically depend on the number of terms $r$ of the equation. While the case $r=1$ is straightforward, the case $r=2$ is already significantly more challenging and is the object of a rich literature, encompassing both direct and iterative methods. Early developments mostly focused on direct methods such as the Bartels and Stewart algorithm \cite{bartels1972solution} and variants thereof, such as Hammarling's method \cite{hammarling1982numerical}. They were later extended to the two-sided version of the same equation in \cite{chu1987solution, gardiner1992solution}. More recently, advanced recursive block splitting strategies were devised in \cite{jonsson2002recursive, jonsson2002blocked} that take advantage of modern computer architecture capabilities.

Since the early 1990s, the focus has drifted towards iterative methods and in particular data-sparse methods that exploit some favorable structure of the solution matrix, including low-rank \cite{saad1989numerical, benner2013low, kressner2010krylov}, rank structured (e.g. quasi-separable) \cite{grasedyck2004existence, massei2018solving} and sparse \cite{palitta2018numerical} formats as well as combinations thereof \cite{palitta2018numerical}. Such methods are necessary for large scale problems (with $n,m \geq 10^5$), when storing the solution matrix becomes impossible. Data-sparse methods represent the solution implicitly; e.g. by storing its low-rank factors. A vast number of methods have been proposed for this purpose, including projection techniques \cite{saad1989numerical, simoncini2007new}, tensorized Krylov subspaces \cite{kressner2010krylov} and quadrature schemes based on integral formula \cite{massei2018solving} to name just a few. For applications related to PDEs, the alternating direction implicit (ADI) method was among the first iterative methods proposed \cite{wachspress1962optimum, wachspress1988iterative}. However, the performance of the method critically depends on a set of parameters whose computation is generally nontrivial and may be as expensive as the iterations themselves. The method has largely been superseded as a standalone solver and is instead commonly used as a preconditioner within specialized versions of well-know Krylov subspace methods such as CG \cite{hestenes1952methods, saad2003iterative}, GMRES \cite{saad1986gmres} or Bi-CGSTAB \cite{van1992bi}. These methods are based on the equivalent Kronecker formulation \cref{eq: kronecker_format} but cleverly exploit the structure of the operator \cite{hochbruck1995preconditioned, jbilou1999global}.

In practice, several of the aforementioned methods are combined to produce very efficient solvers that exploit the structure of the equation and its solution, including the sparsity and (relative) size of the coefficient matrices \cite{simoncini2016computational}. An exhaustive list of methods is beyond the scope of this article and we instead refer to \cite{simoncini2016computational} and the references therein for an overview.

Unfortunately, the picture changes dramatically for $r>2$, especially regarding direct methods. The main reason is that direct solution techniques for $r=2$ rely on joint diagonalization (or triangularization) of matrix pairs such as generalized eigendecompositions (or Schur decompositions), which are also the basis for existence and uniqueness results \cite{chu1987solution}. These techniques generally do not extend to sequences of matrices $\{A_k\}_{k=1}^r$ and $\{B_k\}_{k=1}^r$ (with $r > 2$), unless the elements of these sequences are related in some special way (e.g. they are powers of one same matrix \cite{lancaster1970explicit} or are a commuting family of symmetric matrices). So far, \cref{eq: generalized_sylvester} can only be solved directly in very special cases, e.g. for the so-called \emph{Lyapunov-plus-positive} equation \cref{eq: Lyapunov_plus_positive} with low-rank matrices $N_k$, via the Sherman-Morrison-Woodbury formula \cite{damm2008direct}. However, the method becomes excessively expensive as the rank of $N_k$ grows. To date, the question remains whether the general equation can be solved directly with a target complexity of $O(n^3+m^3)$, without assuming any favorable structure on the coefficient matrices and right-hand side.

In contract, some iterative methods (e.g. projection techniques) extend quite naturally to the generalized multiterm version, although constructing an approximating subspace is nontrivial. Interestingly, such methods sometimes lead to solving a small-size version of the same equation and typically features small dense coefficient matrices \cite{benner2013low, kressner2015truncated}. The lack of general solution techniques even in the small scale case has compelled several authors to explicitly use the Kronecker form for the reduced equation, thereby constraining the rank of the approximate solution to very small integers \cite[Remark 3.1]{kressner2015truncated}. As a matter of fact, solving the reduced equation is even identified as one of the main computational bottlenecks in \cite{kressner2015truncated} and is another argument for considering \cref{eq: generalized_sylvester} in its full generality, without assuming any favorable structure on the coefficient matrices.

In this article, we will focus on devising algebraic parameter-free preconditioning techniques for the iterative solution of \cref{eq: kronecker_format} as a way of solving the generalized multiterm Sylvester equation in \cref{eq: generalized_sylvester}. The methods described herein are applicable as standalone solvers for small to medium size equations and may supplement low-rank solvers for larger ones. Our methods do not assume any specific structure on the coefficient matrices and nearly achieve a target complexity of $O(n^3+m^3)$ for dense unstructured matrices.

The rest of the article is structured as follows: In \cref{se: iterative_methods} we first recall some matrix-oriented Krylov subspace methods for solving \cref{eq: generalized_sylvester}. Similarly to iterative methods for linear systems, these methods may converge very slowly when the associated system matrix in \cref{eq: kronecker_format} is ill-conditioned, creating a formidable strain on computer resources. Therefore, in \cref{se: nkp,se: kinv}, we exploit the underlying Kronecker structure of the system matrix to design efficient and robust preconditioning strategies. These strategies aim at finding low Kronecker rank approximations of the operator itself (\cref{se: nkp}) or its inverse (\cref{se: kinv}). Furthermore, if the inverse admits a good sparse approximation, we propose to combine our strategies with sparse approximate inverse techniques to construct low Kronecker rank sparse approximate inverses. \Cref{se: experiments} illustrates the effectiveness of our preconditioning strategies by comparing them to tailored preconditioners from various applications, including model order reduction, isogeometric analysis and convection-diffusion equations. Finally, \cref{se: conclusion} summarizes our findings and draws conclusions.

\section{Krylov subspace methods for matrix equations}
\label{se: iterative_methods}
Since direct solution methods for \cref{eq: generalized_sylvester} are generally not feasible, we investigate its iterative solution. For this purpose, we denote $\mathcal{M} \colon \mathbb{R}^{m \times n} \to \mathbb{R}^{m \times n}$ the linear operator defined as
\begin{equation}
\label{eq: operator_M}
    \mathcal{M}(X)=\sum_{k=1}^r B_kXA_k^T. 
\end{equation}
This operator has a Kronecker structured matrix representation given by
\begin{equation}
\label{eq: matrix_M}
    M=\sum_{k=1}^r A_k \otimes B_k. 
\end{equation}
We will generally use curly letters for linear operators and straight letters for their associated matrix. Krylov subspace methods based on the Kronecker formulation exploit the fact that
\begin{equation}
\label{eq: operator_application}
     Y=\mathcal{M}(X) \iff \mathbf{y}=M\mathbf{x}
\end{equation}
with $\mathbf{x}=\vectorization(X)$, $\mathbf{y}=\vectorization(Y)$ and never explicitly deal with the Kronecker form in \cref{eq: matrix_M}. The original idea was laid out in \cite{hochbruck1995preconditioned} and revisited several years later in \cite{jbilou1999global}. In essence, these so-called ``global'' Krylov methods are merely specialized implementations of well-known Krylov methods (e.g. CG, GMRES, Bi-CGSTAB,...) applied to the Kronecker formulation. Nevertheless, the computational savings are readily appreciated: if all factor matrices $A_k$ and $B_k$ are dense for $k=1,\dots,r$, storing them requires $O(r(n^2+m^2))$ while applying $\mathcal{M}(X)$ requires $O(r(n^2m+nm^2))$ operations. In comparison, forming $M$ explicitly already requires $O(n^2m^2)$ of storage and $O(n^2m^2)$ for matrix-vector multiplications, thereby constraining $n$ and $m$ to small integers. Matrix-oriented Krylov subspace methods also typically form the backbone of low-rank solvers, where the low-rank structure of the iterates further reduces the computational cost and truncation operators limit the rank growth (see e.g. \cite{kressner2011low, benner2013low} and in particular \cite{simoncini2023analysis} for a recent convergence analysis for CG). Moreover, as already highlighted in \cite{hochbruck1995preconditioned}, preconditioning techniques are straightforwardly incorporated by adapting existing preconditioned Krylov subspace methods for linear systems. In the context of matrix equations, they take the form of a preconditioning operator $\mathcal{P} \colon \mathbb{R}^{m \times n} \to \mathbb{R}^{m \times n}$. Early preconditioning strategies for $r=2$ were based on ADI, symmetric successive overrelaxation (SSOR) and incomplete LU type preconditioning \cite{hochbruck1995preconditioned, bouhamidi2008note}. While ADI remains an important component for preconditioning certain \emph{Lyapunov-plus-positive} equations \cite{benner2013low, benner2013numerical, damm2008direct}, more general techniques are still lacking. The next few sections address this shortcoming.





\section{Nearest Kronecker product preconditioner}
\label{se: nkp}
In \cref{se: introduction}, we had noted that the solution of a generalized Sylvester equation can be computed (relatively) easily when $r \leq 2$. Indeed, for $r=2$ the equation may often be reformulated as a standard Sylvester equation for which there exists dedicated solvers while for $r=1$ the equation reduces to a very simple matrix equation, which can be solved straightforwardly. Therefore, a first preconditioning strategy could rely on finding the best Kronecker rank $1$ or $2$ approximation of $M=\sum_{k=1}^r A_k \otimes B_k$ and use it as a preconditioning operator. Kronecker rank $1$ preconditioners have already been proposed for many different applications including image processing \cite{nagy2006kronecker}, Markov chains \cite{langville2004akronecker, langville2004kronecker, langville2004testing}, stochastic Galerkin \cite{ullmann2010kronecker} and tensorized \cite{gao2013kronecker, gao2014fast, voet2023mathematical} finite element methods. Extensions to Kronecker rank $2$ preconditioners have been considered in \cite{sangalli2016isogeometric, gao2015preconditioners} and also for multiterm Sylvester equations in very specific applications \cite{damm2008direct, benner2013low, palitta2016matrix}. Finding more general, application-independent preconditioners is desirable. The problem of finding the best Kronecker product approximation of a matrix (not necessarily expressed as a sum of Kronecker products) was first investigated by Van Loan and Pitsianis \cite{van1993approximation}. A more modern presentation followed in \cite{golub2013matrix}. We adopt the same general framework for the time being and later specialize it to our problem. For the best Kronecker rank 1 approximation of a matrix $M \in \mathbb{R}^{nm \times nm}$, factor matrices $Y \in \mathbb{R}^{n \times n}$ and $Z \in \mathbb{R}^{m \times m}$ are sought such that $\phi_M(Y,Z)=\|M-Y \otimes Z\|_F$ is minimized. Van Loan and Pitsianis observed that both the Kronecker product $Y \otimes Z$ and $\vectorization(Y)\vectorization(Z)^T$ form all the products $y_{ij}z_{kl}$ for $i,j=1,\dots,n$ and $k,l=1,\dots,m$ but at different locations. Thus, there exists a linear mapping $\mathcal{R} \colon \mathbb{R}^{nm \times nm} \to \mathbb{R}^{n^2 \times m^2}$ (which they called \textit{rearrangement}) such that $\mathcal{R}(Y \otimes Z)=\vectorization(Y)\vectorization(Z)^T$. This mapping is defined explicitly by considering a block matrix $A$ where $A_{ij} \in \mathbb{R}^{m \times m}$ for $i,j=1,\dots,n$. Then, by definition
\begin{equation*}
    A =
    \begin{pmatrix}
    A_{11} & \cdots & A_{1n} \\
    \vdots & \ddots & \vdots \\
    A_{n1} & \cdots & A_{nn}
    \end{pmatrix}
    \qquad
    \mathcal{R}(A)=
    \begin{pmatrix}
    \vectorization(A_{11})^T \\
    \vectorization(A_{21})^T \\
    \vdots \\
    \vectorization(A_{nn})^T
    \end{pmatrix}.
\end{equation*}
By construction, for a matrix $A=Y \otimes Z$,
\begin{equation*}
    Y \otimes Z =
    \begin{pmatrix}
    y_{11}Z & \cdots & y_{1n}Z \\
    \vdots & \ddots & \vdots \\
    y_{n1}Z & \cdots & y_{nn}Z
    \end{pmatrix}
    \qquad \mathcal{R}(Y \otimes Z)=     
    \begin{pmatrix}
    y_{11}\vectorization(Z)^T \\
    y_{21}\vectorization(Z)^T \\
    \vdots \\
    y_{nn}\vectorization(Z)^T
    \end{pmatrix}
    =\vectorization(Y)\vectorization(Z)^T.
\end{equation*}
More generally, since the vectorization operator is linear,  
\begin{equation*}
    \mathcal{R}\left(\sum_{s=1}^r Y_s \otimes Z_s\right)=\sum_{s=1}^r \vectorization(Y_s)\vectorization(Z_s)^T.
\end{equation*}
Therefore, $\mathcal{R}$ transforms a Kronecker rank $r$ matrix into a rank $r$ matrix. Since rearranging the entries of a matrix does not change its Frobenius norm, the minimization problem becomes
\begin{equation}
\label{eq: phi_M_minimization}
    \min \phi_M(Y,Z)=\min \|M-Y \otimes Z\|_F=\min \|\mathcal{R}(M)-\vectorization(Y)\vectorization(Z)^T\|_F.
\end{equation}
Thus, finding the best factor matrices $Y$ and $Z$ is equivalent to finding the best rank 1 approximation of $\mathcal{R}(M)$. More generally, finding the best factor matrices $Y_s$ and $Z_s$ for $s=1,\dots,q$ defining the best Kronecker rank $q$ approximation is equivalent to finding the best rank $q$ approximation of $\mathcal{R}(M)$ and may be conveniently computed with a truncated singular value decomposition (SVD). Computing it is particularly cheap in our context given that $\mathcal{R}(M)$ is already in low-rank format. Note that applying the inverse operator $\mathcal{R}^{-1}$ to the SVD of $\mathcal{R}(M)$ enables to express $M$ as
\begin{equation}
\label{eq: svd_form}
    M=\sum_{k=1}^r \sigma_k (U_k \otimes V_k)
\end{equation}
where $U_k$ and $V_k$ are reshapings of the $k$th left and right singular vectors of $\mathcal{R}(M)$, respectively, and $\sigma_k$ are the singular values for $k=1,\dots,r$. The orthogonality of the left and right singular vectors ensures that $\langle U_i, U_j \rangle_F = \delta_{ij}$, $\langle V_j, V_j \rangle_F = \delta_{ij}$ and $\langle M, (U_i \otimes V_i)\rangle_F=\sigma_i$, where $\langle .,.\rangle_F$ denotes the Frobenius inner product. The best Kronecker rank $q$ approximation $P$ then simply consists in truncating the sum in \cref{eq: svd_form} by retaining its first $q$ leading terms. The approximation error is then given by the tail of the singular values
\begin{equation}
\label{eq: low_rank_approx_error}
    \|M-P\|_F^2=\sum_{k=q+1}^r \sigma_k^2.
\end{equation}
The procedure is summarized in \cref{algo: NKP_preconditioner} and is referred to as the SVD approach. We emphasize that we only consider $q \leq 2$ for constructing a practical preconditioner since $q>2$ would generally be as difficult as solving the original multiterm equation. For $q=1$, the resulting preconditioner is commonly referred to as the nearest Kronecker product preconditioner (NKP) \cite{langville2004akronecker, voet2023mathematical}. We will abusively use the same terminology for $q=2$.

\begin{algorithm}[htbp]
\begin{algorithmic}[1]
\caption{Best Kronecker rank $q$ approximation}
\label{algo: NKP_preconditioner}
\Statex \textbf{Input}:
\Statex Factor matrices $\{A_k\}_{k=1}^r \subset \mathbb{R}^{n \times n}$ and $\{B_k\}_{k=1}^r \subset \mathbb{R}^{m \times m}$
\Statex Kronecker rank $q \leq r$
\Statex \textbf{Output}: 
\Statex Factor matrices $Y_s$ and $Z_s$ for $s=1,\dots,q$ such that $\sum_{s=1}^q Y_s \otimes Z_s \approx \sum_{k=1}^r A_k \otimes B_k$
\Statex
\State Set $V_A=[\vectorization(A_1),\dots,\vectorization(A_r)]$
\State Set $V_B=[\vectorization(B_1),\dots,\vectorization(B_r)]$
\State Compute the thin QR factorization $V_A=Q_AR_A$
\State Compute the thin QR factorization $V_B=Q_BR_B$
\State Compute the SVD $R_AR_B^T=\tilde{U}\Sigma \tilde{V}^T$ \Comment{\parbox[t]{.35\linewidth}{$\Sigma=\diag(\sigma_1,\dots,\sigma_r)$}}
\State Set $V_Y=Q_A\tilde{U}\Sigma^{1/2}$ \Comment{\parbox[t]{.35\linewidth}{$V_Y=[\vectorization(Y_1),\dots,\vectorization(Y_r)]$}}
\State Set $V_Z=Q_B\tilde{V}\Sigma^{1/2}$ \Comment{\parbox[t]{.35\linewidth}{$V_Z=[\vectorization(Z_1),\dots,\vectorization(Z_r)]$}}
\State Return and reshape the first $q$ columns of $V_Y$ and $V_Z$.
\end{algorithmic}
\end{algorithm}

The SVD approach to the best Kronecker product approximation in the Frobenius norm is well established in the numerical linear algebra community. However, Van Loan and Pitsianis also proposed an alternating least squares (ALS) approach, which in our context might be cheaper. We both specialize their strategy to Kronecker rank $r$ matrices and extend it to Kronecker rank $q$ approximations. Adopting the same notations as in \cref{algo: NKP_preconditioner} and employing the reordering $\mathcal{R}$, we obtain
\begin{align}
    \left\|\sum_{k=1}^r A_k \otimes B_k-\sum_{s=1}^q Y_s \otimes Z_s\right\|_F &= \left\|\sum_{k=1}^r \vectorization(A_k)\vectorization(B_k)-\sum_{s=1}^q \vectorization(Y_s)\vectorization(Z_s)\right\|_F \nonumber \\
    &= \|V_AV_B^T-V_YV_Z^T\|_F \label{eq: als_optimization}
\end{align}
If the matrices $Z_s$ are fixed for $s=1,\dots,q$, the optimal solution of the least squares problem \cref{eq: als_optimization} is given by
\begin{equation}
    V_Y=V_AV_B^TV_Z(V_Z^TV_Z)^{-1}. \label{eq: als_solution_Y}
\end{equation}
If instead all matrices $Y_s$ are fixed for $s=1,\dots,q$, the optimal solution of \cref{eq: als_optimization} is given by the similar looking expression
\begin{equation}
    V_Z=V_BV_A^TV_Y(V_Y^TV_Y)^{-1}. \label{eq: als_solution_Z}
\end{equation}
The inverse in \cref{eq: als_solution_Y} (resp. \cref{eq: als_solution_Z}) exists provided the matrices $Z_s$ (resp. $Y_s$) are linearly independent. Equations \cref{eq: als_solution_Y} and \cref{eq: als_solution_Z} reveal that all factor matrices $Y_s$ and $Z_s$ for $s=1,\dots,q$ are linear combinations of $A_k$ and $B_k$, respectively, which could already be inferred from the SVD approach. This finding was already stated in \cite{van1993approximation} and proved in \cite[Theorem 4.1]{langville2004akronecker} for $q=1$ and \cite[Theorem 4.2]{voet2023mathematical} for arbitrary $q$. In particular, for $q=1$, after some reshaping, equations \cref{eq: als_solution_Y} and \cref{eq: als_solution_Z} reduce to
\begin{equation*}
    Y=\sum_{k=1}^r \frac{\langle B_k, Z \rangle_F}{\langle Z,Z \rangle_F}A_k \quad \text{and} \quad Z=\sum_{k=1}^r \frac{\langle A_k, Y \rangle_F}{\langle Y,Y \rangle_F}B_k,
\end{equation*}
respectively, which can also be deduced from \cite[Theorem 4.1]{van1993approximation}. Our derivations are summarized in \cref{algo: als_rank_q_approximation}. The norm of the residual is used as stopping criterion in the alternating least squares algorithm. It can be cheaply evaluated without forming the Kronecker products explicitly since
\begin{align*}
    \|V_AV_B^T-V_YV_Z^T\|_F^2&=\|V_AV_B^T\|_F^2-2\langle V_AV_B^T, V_YV_Z^T \rangle_F+\|V_YV_Z^T\|_F^2 \\
    &=\langle V_A^TV_A, V_B^TV_B \rangle_F - 2 \langle V_A^TV_Y, V_B^TV_Z \rangle_F + \langle V_Y^TV_Y, V_Z^TV_Z \rangle_F.
\end{align*}
A more explicit expression already appeared in \cite[Theorem 4.2]{langville2004akronecker} for $r=2$ and $q=1$. Our expression generalizes it to arbitrary $r$ and $q$.

\begin{algorithm}[htbp]
\begin{algorithmic}[1]
\caption{ALS for Kronecker rank $q$ approximation}
\label{algo: als_rank_q_approximation}
\Statex \textbf{Input}:
\Statex Factor matrices $\{A_k\}_{k=1}^r \subset \mathbb{R}^{n \times n}$ and $\{B_k\}_{k=1}^r \subset \mathbb{R}^{m \times m}$
\Statex Linearly independent factor matrices $Z_s \in \mathbb{R}^{m \times m}$ for $s=1,\dots,q$
\Statex Tolerance $\epsilon>0$ and maximum number of iterations $N \in \mathbb{N}$
\Statex \textbf{Output}: 
\Statex Factor matrices $Y_s$ and $Z_s$ for $s=1,\dots,q$ such that $\sum_{s=1}^q Y_s \otimes Z_s \approx \sum_{k=1}^r A_k \otimes B_k$
\Statex
\State Set $V_A=[\vectorization(A_1),\dots,\vectorization(A_r)]$, \Comment{\parbox[t]{.25\linewidth}{Initialization}}
\State Set $V_B=[\vectorization(B_1),\dots,\vectorization(B_r)]$,
\State Set $V_Z=[\vectorization(Z_1),\dots,\vectorization(Z_q)]$,
\State Set $r=\infty$, $j=0$ 
\While{$\sqrt{r}>\epsilon$ \textbf{ and } $j \leq N$}
\State Compute $V_Y=V_AV_B^TV_Z(V_Z^TV_Z)^{-1}$ \Comment{\parbox[t]{.25\linewidth}{Optimizing for $Y$}}
\State Compute $V_Z=V_BV_A^TV_Y(V_Y^TV_Y)^{-1}$ \Comment{\parbox[t]{.25\linewidth}{Optimizing for $Z$}}
\Statex
\State Compute $r=\langle V_A^TV_A, V_B^TV_B \rangle_F - 2 \langle V_A^TV_Y, V_B^TV_Z \rangle_F + \langle V_Y^TV_Y, V_Z^TV_Z \rangle_F$
\State Update $j=j+1$
\EndWhile
\State Return $Y_s$ and $Z_s$ for $s=1,\dots,q$.
\end{algorithmic}
\end{algorithm}

\subsection{Complexity analysis}
We briefly compare the complexity of both algorithms. For \cref{algo: NKP_preconditioner}, the QR factorizations in lines 3 and 4 require about $O(r^2(n^2+m^2))$ flops (if $r \ll n,m$) \cite{golub2013matrix, demmel1997applied}. Computing the SVD in line 5 only requires $O(r^3)$ while the matrix-matrix products in lines 6 and 7 require $O(r^2(n^2+m^2))$.

For \cref{algo: als_rank_q_approximation}, if $r \ll n,m$, lines 6 and 7 require about $O(rq(n^2+m^2))$ operations. Naively recomputing the residual at each iteration in line 8 may be quite costly. Therefore, we suggest computing $\langle V_A^TV_A, V_B^TV_B \rangle_F$ once for $O(r^2(n^2+m^2))$ flops and storing the result. The last two terms of the residual can be cheaply evaluated if intermediate computations necessary in lines 6 and 7 are stored. Therefore, for $N$ iterations, the total cost amounts to $O((r^2+Nrq)(n^2+m^2))$ and is quite similar to the SVD framework if $N$ and $q$ remain small.

\subsection{Theoretical results}
The approximation problem in the Frobenius norm is mainly motivated for computational reasons. However, it also offers some theoretical guarantees, which are summarized in this section. The next theorem first recalls a very useful result for Kronecker rank $1$ approximations.

\begin{theorem}[{\cite[Theorems 5.1, 5.3 and 5.8]{van1993approximation}}]
\label{th: kr_1_approximation}
Let $M \in \mathbb{R}^{nm \times nm}$ be a block-banded, nonnegative and symmetric positive definite matrix. Then, there exists banded, nonnegative and symmetric positive definite factor matrices $Y$ and $Z$ such that $\phi_M(Y,Z)$ in \cref{eq: phi_M_minimization} is minimized.
\end{theorem}

Thus, the properties of $M$ are inherited by its approximation $Y \otimes Z$.
However, not all properties of \cref{th: kr_1_approximation} extend to Kronecker rank $q \geq 2$. Clearly, due to the orthogonality relations $\langle U_i, U_j \rangle_F = \delta_{ij}$, $\langle V_j, V_j \rangle_F = \delta_{ij}$ deduced from the SVD approach, only $Y_1$ and $Z_1$ are nonnegative if $M$ is. However, other useful properties such as sparsity and symmetry are preserved. We formalize it through the following definition.

\begin{definition}[Sparsity pattern]
\label{def: sparsity pattern}
The sparsity pattern of a matrix $A \in \mathbb{R}^{n \times n}$ is the set
\begin{equation*}
    \sparsity(A)=\{(i,j) \colon a_{ij} \neq 0, 1 \leq i,j \leq n\}
\end{equation*}
\end{definition}

The following lemma summarizes some useful properties shared by the SVD and alternating least squares solutions. Its proof is an obvious consequence of \cref{algo: NKP_preconditioner,algo: als_rank_q_approximation}.

\begin{lemma}
\label{lem: als_approximation_prop}
Let $\sum_{s=1}^q Y_s \otimes Z_s$ be the Kronecker rank $q$ approximation computed with \cref{algo: NKP_preconditioner} or \cref{algo: als_rank_q_approximation}. Then,
\begin{itemize}[noitemsep]
    \item If all $A_k$ and $B_k$ are symmetric, then all $Y_s$ and $Z_s$ also are.
    \item The sparsity patterns of $Y_s$ and $Z_s$ are contained in those of $A_k$ and $B_k$; i.e. 
    \begin{equation*}
        \sparsity(Y_s) \subseteq \bigcup_{k=1}^r \sparsity(A_k), \qquad  \sparsity(Z_s) \subseteq \bigcup_{k=1}^r \sparsity(B_k) \qquad s=1,\dots,q.
    \end{equation*}
\end{itemize}
\end{lemma}
Note that the properties listed in \cref{lem: als_approximation_prop} do not depend on the initial guesses.

In this work, we are interested in computing Kronecker product approximations as a means of constructing efficient preconditioners. Therefore, we would like to connect the approximation quality to the preconditioning effectiveness. Several authors have attempted to obtain estimates for the condition number of the preconditioned system or some related measure \cite{ullmann2010kronecker, voet2023mathematical}. We present hereafter a general result, which is only satisfactory for small or moderate condition numbers of $M$.

\begin{lemma}
\label{lem: NKP_preconditioning}
Let $M, \tilde{M} \in \mathbb{R}^{n \times n}$ be symmetric positive definite matrices. Then,
\begin{equation*}
    \frac{1}{\kappa(M)}\frac{\|M-\tilde{M}\|_F}{\|M\|_F} \leq \sqrt{\frac{1}{n}\sum_{i=1}^n \left(1-\frac{1}{\lambda_i(M,\tilde{M})}\right)^2} \leq \kappa(M)\frac{\|M-\tilde{M}\|_F}{\|M\|_F}
\end{equation*}
where $\kappa(M)=\frac{\lambda_n(M)}{\lambda_1(M)}$ is the spectral condition number of $M$.
\end{lemma}
\begin{proof}
Consider the matrix pair $(M, \tilde{M})$. Since $M$ and $\tilde{M}$ are symmetric positive definite, there exists an invertible matrix $U \in \mathbb{R}^{n \times n}$ of $\tilde{M}$-orthonormal eigenvectors such that $U^TMU=D$ and $U^T\tilde{M}U=I$, where $D=\diag(\lambda_1,\dots,\lambda_n)$ is the diagonal matrix of positive eigenvalues \cite[Theorem VI.1.15]{stewart1990matrix}. Now note that
\begin{equation*}
    \sqrt{\frac{1}{n}\sum_{i=1}^n \left(1-\lambda_i(M,\tilde{M})\right)^2} = \frac{\|I-D\|_F}{\|I\|_F} = \frac{\|U^T(M-\tilde{M})U\|_F}{\|U^T\tilde{M}U\|_F}.
\end{equation*}
Moreover,
\begin{equation*}
    \|U\|_2^{-2}\|U^{-1}\|_2^{-2} \frac{\|M-\tilde{M}\|_F}{\|\tilde{M}\|_F} \leq \frac{\|U^T(M-\tilde{M})U\|_F}{\|U^T\tilde{M}U\|_F} \leq \|U\|_2^2\|U^{-1}\|_2^2 \frac{\|M-\tilde{M}\|_F}{\|\tilde{M}\|_F}.
\end{equation*}
The quantity $\kappa(U)^2=\|U\|_2^2\|U^{-1}\|_2^2$ appearing in the bounds is nothing more than the condition number of $\tilde{M}$. Indeed, thanks to the normalization of the eigenvectors $\tilde{M}=U^{-T}U^{-1}$ and $\tilde{M}^{-1}=UU^T$. Consequently, $\|U^{-1}\|_2^2=\|U^{-T}U^{-1}\|_2=\|\tilde{M}\|_2$ and $\|U\|_2^2=\|UU^T\|_2=\|\tilde{M}^{-1}\|_2$. Finally, we obtain the bounds
\begin{equation}
    \frac{1}{\kappa(\tilde{M})}\frac{\|M-\tilde{M}\|_F}{\|\tilde{M}\|_F} \leq \sqrt{\frac{1}{n}\sum_{i=1}^n \left(1-\lambda_i(M,\tilde{M})\right)^2} \leq \kappa(\tilde{M})\frac{\|M-\tilde{M}\|_F}{\|\tilde{M}\|_F}. \label{eq: preliminary_bound}
\end{equation}
Since the eigenvalues of $(\tilde{M},M)$ are the reciprocal of the eigenvalues of $(M,\tilde{M})$, we conclude by swapping the roles of $M$ and $\tilde{M}$.
\end{proof}

\begin{remark}
If $\tilde{M}$ is the best Kronecker rank $q$ approximation, then, following \cref{eq: low_rank_approx_error}, we obtain the more explicit upper bound
\begin{equation*}
    \sqrt{\frac{1}{n}\sum_{i=1}^n \left(1-\frac{1}{\lambda_i(M,\tilde{M})}\right)^2} \leq \kappa(M)\sqrt{\sum_{k=q+1}^r \left(\frac{\sigma_k}{\sigma_1}\right)^2}.
\end{equation*}
The upper bound in particular depends on the ratio of singular values, which was already suspected by some authors \cite{langville2004testing, voet2023mathematical} but to our knowledge never formally proved. \Cref{lem: NKP_preconditioning} indicates that the nearest Kronecker product preconditioners might be very effective if $\mathcal{R}(M)$ features fast decaying singular values.
\end{remark}

\section{Low Kronecker rank approximate inverse}
\label{se: kinv}
Instead of finding an approximation of the operator itself, we will now find an approximation of its inverse. Clearly, since $(A \otimes B)^{-1}=A^{-1} \otimes B^{-1}$ for invertible matrices $A,B$ \cite[Corollary 4.2.11]{horn1991topics}, (invertible) Kronecker rank $1$ matrices have a Kronecker rank $1$ inverse. However, the relation between the Kronecker rank of a matrix and the Kronecker rank of its inverse is not obvious for $r \geq 2$. Although the latter could be much larger than the former, the inverse might still be very well approximated by low Kronecker rank matrices. Indeed, it was shown in \cite{grasedyck2004existence} that the inverse of sums of Kronecker products obtained by finite difference and finite element discretizations of some model problems can be well approximated by Kronecker products of matrix exponentials (exponential sums). Unfortunately, due to the special tensor product structure, these results are limited to idealized problems rarely met in practice. Nevertheless, these insightful results suggest the possibility of generally approximating the inverse of an arbitrary sum of Kronecker products by a low Kronecker rank matrix. We will describe in this section a general and algebraic way of constructing such an approximation without ever forming the Kronecker product matrix explicitly. We first consider the rank 1 case and later extend it to rank $q \geq 2$.

\subsection{Kronecker rank $1$ approximate inverse}
\label{se: rank_1_kinv}
We set at finding factor matrices $C \in \mathbb{R}^{n \times n}$ and $D \in \mathbb{R}^{m \times m}$ such that $C \otimes D \approx \left(\sum_{k=1}^r A_k \otimes B_k\right)^{-1}$ and therefore consider the minimization problem
\begin{equation}
\label{eq: als}
   \min_{C,D} \|I-\sum_{k=1}^r A_kC \otimes B_kD\|_F
\end{equation}
where we have used the mixed-product property of the Kronecker product (see e.g. \cite[Lemma 4.2.10]{horn1991topics}). The minimization problem is nonlinear when optimizing for $(C,D)$ simultaneously, but is linear when optimizing for $C$ or $D$ individually. This observation motivates an alternating optimization approach and is based on solving a sequence of linear least squares problems. Assume for the time being that $C$ is fixed and $D$ must be computed. Since any permutation or matrix reshaping is an isometry in the Frobenius norm, the block matrices
\begin{equation}
\label{eq: reshaping}
M=
\begin{pmatrix}
M_{11} & \hdots & M_{1n} \\
\vdots & \ddots & \vdots \\
M_{n1} & \hdots & M_{nn}
\end{pmatrix}
\quad \text{and} \quad
\tilde{M}=
\begin{pmatrix}
    M_{11} \\
    M_{21} \\
    \vdots \\
    M_{nn}
\end{pmatrix}
\end{equation}
have the same Frobenius norm. Applying this transformation to \cref{eq: als}, we obtain
\begin{equation}
\label{eq: transformation}
    \|I-\sum_{k=1}^r A_kC \otimes B_kD\|_F=\|\tilde{I}-\sum_{k=1}^r \vectorization(A_kC) \otimes B_kD\|_F=\|\tilde{I}-(U \otimes I_m)BD\|_F
\end{equation}
where 
\begin{equation}
\tilde{I}=[I_m; 0; \dots; 0; I_m], \quad B=[B_1; B_2; \dots; B_r],
\label{eq: definition_It_B}
\end{equation}
and we have defined $U=[\vectorization(A_1C), \dots, \vectorization(A_rC)] \in \mathbb{R}^{n^2 \times r}$. The semi-colon in \cref{eq: definition_It_B} means that the factor matrices are stacked one above the other. Minimizing the Frobenius norm in \cref{eq: transformation} for the matrix $D$ is indeed equivalent to solving a linear least squares problem for each column of $D$ with coefficient matrix $\mathcal{B}=(U \otimes I_m)B=\sum_{k=1}^r \vectorization(A_kC) \otimes B_k$ of size $mn^2 \times m$. For obvious storage reasons, we will never form this matrix explicitly (which would be as bad as forming the Kronecker product explicitly). The QR factorization is very efficient for solving least squares problems involving Kronecker products \cite{fausett1994large} but unfortunately, $\mathcal{B}$ is the \emph{product} of two matrices and only one of them is a Kronecker product. It is unclear how this structure may be leveraged. Fortunately, forming and solving the normal equations instead is very appealing because of its ability to compress large least squares problems into much smaller linear systems. Indeed,
\begin{equation*}
    \mathcal{B}^T\mathcal{B}=B^T(U^TU \otimes I_m)B=\sum_{k,l=1}^r \vectorization(A_kC)^T \vectorization(A_lC) B_k^TB_l=\sum_{k,l=1}^r \beta_{kl} B_k^TB_l
\end{equation*}
with $\beta_{kl}=\vectorization(A_kC)^T \vectorization(A_lC) \in \mathbb{R}$ for $k,l=1,\dots,r$. Therefore, $\mathcal{B}^T\mathcal{B}$ has size $m$, independently of the Kronecker rank $r$. The right-hand side of the normal equations is $\mathcal{B}^T\tilde{I}$. Thanks to the structure of $\tilde{I}$, the computation of this term can be drastically simplified. For a general matrix $\tilde{M}$, as defined in \cref{eq: reshaping}, we have
\begin{equation}
\label{eq: rhs}
    \mathcal{B}^T\tilde{M}=\left(\sum_{k=1}^r \vectorization(A_kC)^T \otimes B_k^T\right)\tilde{M}=\sum_{k=1}^r \sum_{i,j=1}^n (A_kC)_{ij}B_k^TM_{ij}.
\end{equation}
However, for $\tilde{M}=\tilde{I}$, we have $M_{ii}=I_m$ for $i=1,\dots,n$ and $M_{ij}=0$ for all $i \neq j$. Thus, \cref{eq: rhs} reduces to
\begin{equation*}
    \mathcal{B}^T\tilde{I}=\sum_{k=1}^r B_k^T \sum_{i=1}^n (A_kC)_{ii} = \sum_{k=1}^r \trace(A_kC)B_k^T=\sum_{k=1}^r \delta_k B_k^T.
\end{equation*}
with coefficients $\delta_k=\trace(A_kC)$. Note that the coefficients $\beta_{kl}$ can also be expressed as
\begin{equation*}
    \beta_{kl}=\vectorization(A_kC)^T \vectorization(A_lC)=\langle A_kC, A_lC\rangle_F=\langle A_k^TA_l,CC^T\rangle_F,
\end{equation*}
while the coefficients $\delta_k$ are given by
\begin{equation*}
    \delta_k=\trace(A_kC)=\langle A_k^T, C \rangle_F.
\end{equation*}
Although the factors $\beta_{kl}$ and $\delta_k$ may seem related, it must be emphasized that $\beta_{kl} \neq \delta_k \delta_l$. Indeed, $\beta_{kl}$ involves all entries of $A_kC$ and $A_lC$ whereas $\delta_k \delta_l$ only involves their diagonal entries. As a matter of fact,
\begin{align*}
    \trace(A_kC \otimes A_lC)&=\trace(A_kC)\trace(A_lC)=\delta_k\delta_l, \\
    \trace(\mathcal{R}(A_kC \otimes A_lC))&=\trace(\vectorization(A_kC)\vectorization(A_lC)^T)=\vectorization(A_kC)^T\vectorization(A_lC)=\beta_{kl}.
\end{align*}
Since the factor matrices $A_k$ and $B_k$ for $k=1,\dots,r$ do not change during the course of the iterations, if $r$ is relatively small it might be worthwhile precomputing the products $A_k^TA_l$ and $B_k^TB_l$ for $k,l=1,\dots,r$ at the beginning of the algorithm. Storing these matrices will require $O(r^2(n^2+m^2))$ of memory. Provided $r$ is small with respect to $n$ and $m$, the memory footprint is still significantly smaller than the $O(n^2m^2)$ required for storing the Kronecker product matrix explicitly.

We now assume that $D$ is fixed and $C$ must be computed. For this purpose, we recall that there exists a perfect shuffle permutation matrix $S_{n,m}$ \cite[Corollary 4.3.10]{horn1991topics} such that
\begin{equation*}
    S_{n,m}(A \otimes B)S_{n,m}^T = B \otimes A.
\end{equation*}
Since permutation matrices are orthogonal and the Frobenius norm is unitarily invariant,
\begin{equation*}
    \|I-\sum_{k=1}^r A_kC \otimes B_kD\|_F=\|S_{n,m}(I-\sum_{k=1}^r A_kC \otimes B_kD)S_{n,m}^T\|_F=\|I-\sum_{k=1}^r B_kD \otimes A_kC\|_F.
\end{equation*}
Therefore, the expressions when optimizing for $C$ are completely analogous, with $B_k$ swapped for $A_k$ and $C$ swapped for $D$. We define
\begin{align*}
    \mathcal{A}^T\mathcal{A}&=\sum_{k,l=1}^r \alpha_{kl} A_k^TA_l, & \alpha_{kl}&=\langle B_k^TB_l,DD^T\rangle_F, \\
    \mathcal{A}^T\tilde{I}&=\sum_{k=1}^r \gamma_k A_k^T, & \gamma_k&=\langle B_k^T, D \rangle_F.
\end{align*}
Note that $\tilde{I}$ is here defined by applying the transformation \cref{eq: reshaping} to $I_m \otimes I_n$ (and not $I_n \otimes I_m$ as in \cref{eq: definition_It_B}). Its size is $m^2n \times n$ and its only nontrivial blocks are identity matrices of size $n$. With a slight abuse of notation, we will not distinguish the two reshaped identity matrices since it will always be clear from the context which one is used.

The stopping criterion of the alternating least squares algorithm relies on evaluating the residual at each iteration. If this operation is done naively, much of the computational saving is lost in addition to prohibitive memory requirements. Fortunately, the residual may be evaluated at negligible additional cost by recycling quantities that were previously computed:

\begin{align}
\label{eq: residual}
    &\|I-\sum_{k=1}^r A_kC \otimes B_kD\|_F^2 = nm-2 \langle I, \sum_{k=1}^r A_kC \otimes B_kD \rangle_F + \left\|\sum_{k=1}^r A_kC \otimes B_kD\right\|_F^2 \nonumber \\
    &=nm-2\trace\left(\sum_{k=1}^r A_kC \otimes B_kD\right)+\trace\left(\sum_{k,l=1}^r A_kCC^TA_l^T \otimes B_kDD^TB_l^T\right)\nonumber \\
    &=nm-2\sum_{k=1}^r \trace(A_kC)\trace(B_kD) + \sum_{k,l=1}^r\trace(A_kCC^TA_l^T)\trace(B_kDD^TB_l^T) \nonumber \\
    &=nm-2\sum_{k=1}^r \gamma_k \delta_k +\sum_{k,l=1}^r \alpha_{kl}\beta_{kl}.
\end{align}
Since the scalars $\alpha_{kl}, \beta_{kl}, \gamma_{k}$ and $\delta_{k}$ have already been computed, evaluating \cref{eq: residual} nearly comes for free as a byproduct of the ALS iterations. The entire procedure is summarized in \cref{algo: rank_1_approximate_inverse}.

\begin{algorithm}[htbp]
\begin{algorithmic}[1]
\caption{ALS for Kronecker rank $1$ approximate inverse}
\label{algo: rank_1_approximate_inverse}
\Statex \textbf{Input}:
\Statex Factor matrices $\{A_k\}_{k=1}^r \subset \mathbb{R}^{n \times n}$ and $\{B_k\}_{k=1}^r \subset \mathbb{R}^{m \times m}$
\Statex Initial guess for the factor matrix $C \in \mathbb{R}^{n \times n}$
\Statex Tolerance $\epsilon>0$ and maximum number of iterations $N \in \mathbb{N}$
\Statex \textbf{Output}: 
\Statex Factor matrices $C$ and $D$ such that $C \otimes D \approx \left(\sum_{k=1}^r A_k \otimes B_k\right)^{-1}$
\Statex
\State Set $r=\infty$, $j=0$ \Comment{\parbox[t]{.25\linewidth}{Initialization}}
\While{$\sqrt{r}>\epsilon$ \textbf{ and } $j \leq N$}
\Statex \Comment{\parbox[t]{.25\linewidth}{Optimizing for $D$}}
\State Compute $\beta_{kl}=\langle A_k^TA_l,CC^T\rangle_F$ for $k,l=1,\dots,r$ \Comment{\parbox[t]{.25\linewidth}{$O(rn^3+r^2n^2)$}}
\State Compute $\delta_{k}=\langle A_k^T, C \rangle_F$ for $k=1,\dots,r$ \Comment{\parbox[t]{.25\linewidth}{$O(rn^2)$}}
\State Form $\mathcal{B}^T\mathcal{B}=\sum_{k,l=1}^r \beta_{kl} B_k^TB_l$ \Comment{\parbox[t]{.25\linewidth}{$O(rm^3+r^2m^2)$}}
\State Form $\mathcal{B}^T\tilde{I}=\sum_{k=1}^r \delta_k B_k^T$ \Comment{\parbox[t]{.25\linewidth}{$O(rm^2)$}}
\State Solve $\mathcal{B}^T\mathcal{B}D=\mathcal{B}^T\tilde{I}$ \Comment{\parbox[t]{.25\linewidth}{$O(m^3)$}}
\Statex \Comment{\parbox[t]{.25\linewidth}{Optimizing for $C$}}
\State Compute $\alpha_{kl}=\langle B_k^TB_l,DD^T\rangle_F$ for $k,l=1,\dots,r$ \Comment{\parbox[t]{.25\linewidth}{$O(rm^3+r^2m^2)$}}
\State Compute $\gamma_{k}=\langle B_k^T, D \rangle_F$ for $k=1,\dots,r$ \Comment{\parbox[t]{.25\linewidth}{$O(rm^2)$}}
\State Form $\mathcal{A}^T\mathcal{A}=\sum_{k,l=1}^r \alpha_{kl} A_k^TA_l$ \Comment{\parbox[t]{.25\linewidth}{$O(rn^3+r^2n^2)$}}
\State Form $\mathcal{A}^T\tilde{I}=\sum_{k=1}^r \gamma_k A_k^T$ \Comment{\parbox[t]{.25\linewidth}{$O(rn^2)$}}
\State Solve $\mathcal{A}^T\mathcal{A}C=\mathcal{A}^T\tilde{I}$ \Comment{\parbox[t]{.25\linewidth}{$O(n^3)$}}
\Statex
\State Update $\beta_{kl}$ and $\delta_{k}$ following lines 3 and 4, respectively \Comment{\parbox[t]{.25\linewidth}{Residual}}
\State Compute $r=nm-2\sum_{k=1}^r \gamma_k \delta_k +\sum_{k,l=1}^r \alpha_{kl}\beta_{kl}$ \Comment{\parbox[t]{.25\linewidth}{$O(r^2)$}}
\State Update $j=j+1$
\EndWhile
\State Return $C$ and $D$
\end{algorithmic}
\end{algorithm}

\subsubsection{Complexity analysis}
When presenting \cref{algo: rank_1_approximate_inverse}, we have favored clarity over efficiency. A practical implementation might look very different and we now describe in detail the tricks that are deployed to reduce its complexity. Since the algorithmic steps for $C$ and $D$ are similar, we only discuss those for $D$ and later adapt them to $C$. \textbf{For simplicity, we will assume that all factor matrices $A_k$ and $B_k$ are dense}.
\begin{itemize}
    \item In line 3, an alternative expression for $\beta_{kl}$
    \begin{equation*}
        \beta_{kl}=\langle A_k^TA_l,CC^T\rangle_F=\langle A_kC, A_lC\rangle_F
    \end{equation*}
    immediately reveals the symmetry ($\beta_{kl}=\beta_{lk}$). Thus, only $\frac{r}{2}(r+1)$ coefficients must be computed, instead of $r^2$. Moreover, their computation only requires $r$ matrix-matrix products $A_kC$ for $k=1,\dots,r$ and then a few Frobenius inner products, which in total amount to $O(rn^3+r^2n^2)$ operations.
    \item A naive implementation of line 5 would require $r^2$ matrix-matrix products. This number can be reduced significantly thanks to the sum factorization technique. After rewriting the equation as
    \begin{equation*}
        \mathcal{B}^T\mathcal{B}=\sum_{k,l=1}^r \beta_{kl} B_k^TB_l=\sum_{k=1}^r B_k^T \sum_{l=1}^r \beta_{kl}B_l,
    \end{equation*}
    we notice that only $r$ matrix-matrix products are needed once all matrices $\sum_{l=1}^r \beta_{kl}B_l$ for $k=1,\dots,r$ have been computed. This technique trades some matrix-matrix products for a few additional (but cheaper) matrix sums. The workload in this step amounts to $O(rm^3+r^2m^2)$ operations.
    \item Since all coefficients are independent, the algorithm is well suited for parallel computations and a suitable sequencing of operations avoids updating $\beta_{kl}$ and $\delta_{k}$ before evaluating the residual.
\end{itemize}
Computing the coefficients $\delta_k$ and forming $\mathcal{B}^T\tilde{I}$ is significantly cheaper and only leads to low order terms, which are neglected. Finally, solving the normal equations in line 7 with a standard direct solver will require $O(m^3)$ operations. After performing a similar analysis for the optimization of $C$ and assuming that $N$ iterations of the algorithm were necessary, the final cost amounts to $O(Nr(n^3+m^3)+Nr^2(n^2+m^2))$. The cost for evaluating the residual is negligible and does not enter our analysis. For the sake of completeness, the cost of each step is summarized in \cref{algo: rank_1_approximate_inverse}. It may often be reduced if the factor matrices are sparse. Note in particular that the sparsity pattern of the system matrix of the normal equations does not change during the course of the iterations. Therefore, sparse direct solvers only require a single symbolic factorization. Moreover, forming the normal equations benefits from highly optimized matrix-matrix multiplication algorithms (level 3 BLAS) available in common scientific computing environments.


\subsection{Kronecker rank $q$ approximate inverse}
\label{se: rank_q_kinv}
If the inverse does not admit a good Kronecker product approximation, the result of \cref{algo: rank_1_approximate_inverse} may be practically useless. To circumvent this issue, it might be worthwhile looking for approximations having Kronecker rank $q \geq 2$. We will see in this section how our strategies developed for rank $1$ approximations may be extended to rank $q \geq 2$. We therefore consider the problem of finding $C_s \in \mathbb{R}^{n \times n}$ and $D_s \in \mathbb{R}^{m \times m}$ for $s=1,\dots,q$ that minimize
\begin{equation*}
    \|I-\sum_{s=1}^q\sum_{k=1}^r A_kC_s \otimes B_kD_s\|_F.
\end{equation*}
For the rank $1$ case, we had first transformed the problem to an equivalent one by stacking all the blocks of the matrix one above the other in reverse lexicographical order. In order to use the same transformation for the rank $q$ case, we must first find an expression for the $(i,j)$th block of $\sum_{s=1}^q\sum_{k=1}^r A_kC_s \otimes B_kD_s$. This can be conveniently done by applying the same strategy adopted earlier. Indeed, the $(i,j)$th block of the matrix is
\begin{equation*}
    \sum_{s=1}^q\sum_{k=1}^r (A_kC_s)_{ij} B_kD_s = \left[\sum_{k=1}^r (A_kC_1)_{ij} B_k, \dots, \sum_{k=1}^r (A_kC_q)_{ij} B_k\right]D
\end{equation*}
where $D=[D_1; \dots; D_q]$. After stacking all the blocks $(i,j)$ for $i,j=1,\dots,n$ on top of each other, we deduce the coefficient matrix for the least squares problem
\begin{equation}
\label{eq: coeff_mat_ls}
    \mathcal{B}=[(U_1 \otimes I_m)B, \dots, (U_q \otimes I_m)B] \in \mathbb{R}^{n^2m \times qm}
\end{equation}
where $U_s=[\vectorization(A_1C_s), \dots, \vectorization(A_rC_s)] \in \mathbb{R}^{n^2 \times r}$ for $s=1,\dots,q$ and $B$ is the same as defined in \cref{eq: definition_It_B} for the rank $1$ approximation. Once again, the matrix $\mathcal{B}$ will never be formed explicitly and we will instead rely on the normal equations. Although the size of the problem is larger, its structure is very similar to the rank $1$ case. Indeed $\mathcal{B}^T\mathcal{B} \in \mathbb{R}^{qm \times qm}$ is a $q \times q$ block matrix consisting of blocks of size $m \times m$. The $(s,t)$th block is given by
\begin{equation*}
    (\mathcal{B}^T\mathcal{B})_{st}=B^T(U_s^TU_t \otimes I_m)B=\sum_{k,l=1}^r \vectorization(A_kC_s)^T \vectorization(A_lC_t) B_k^TB_l=\sum_{k,l=1}^r \beta_{kl}^{st} B_k^TB_l
\end{equation*}
where we have defined $\beta_{kl}^{st}=\vectorization(A_kC_s)^T \vectorization(A_lC_t)=\langle A_k^TA_l,C_sC_t^T\rangle_F$.
The steps for the right-hand side are analogous: $\mathcal{B}^T\tilde{I} \in \mathbb{R}^{qm \times m}$ is a $q \times 1$ block matrix and its $s$th block is given by
\begin{equation*}
    B^T(U_s^T \otimes I_m)\tilde{I}=\sum_{k=1}^r B_k^T \sum_{i=1}^n (A_kC_s)_{ii} = \sum_{k=1}^r \trace(A_kC_s)B_k^T=\sum_{k=1}^r \delta_k^s B_k^T.
\end{equation*}
with $\delta_k^s=\langle A_k^T, C_s \rangle_F$. We further note that $\mathcal{B}^T\mathcal{B}$ and $\mathcal{B}^T\tilde{I}$ can be expressed as
\begin{equation*}
    \mathcal{B}^T\mathcal{B} = \sum_{k,l=1}^r  b_{kl} \otimes B_k^TB_l, \quad \mathcal{B}^T\tilde{I}=\sum_{k=1}^r d_k \otimes B_k^T
\end{equation*}
with
\begin{equation}
b_{kl}=
\begin{pmatrix}
    \beta_{kl}^{11} & \dots & \beta_{kl}^{1q} \\
    \vdots & \ddots & \vdots \\
    \beta_{kl}^{q1} & \dots & \beta_{kl}^{qq}  
\end{pmatrix} \quad \text{and} \quad 
d_k=
\begin{pmatrix}
    \delta_k^1 \\
    \vdots \\
    \delta_k^q
\end{pmatrix}.
\label{eq: b_kl_d_k}
\end{equation}
Resorting to perfect shuffle permutations allows to write a similar least squares problem for $C=[C_1; \dots; C_q]$ once the coefficient matrices $D_s$ for $s=1,\dots,q$ have been computed. It leads to defining the quantities
\begin{equation*}
    \mathcal{A}^T\mathcal{A} = \sum_{k,l=1}^r  a_{kl} \otimes A_k^TA_l, \quad \mathcal{A}^T\tilde{I}=\sum_{k=1}^r c_k \otimes A_k^T
\end{equation*}
with
\begin{equation}
a_{kl}=
\begin{pmatrix}
    \alpha_{kl}^{11} & \dots & \alpha_{kl}^{1q} \\
    \vdots & \ddots & \vdots \\
    \alpha_{kl}^{q1} & \dots & \alpha_{kl}^{qq}  
\end{pmatrix}, \quad 
c_k=
\begin{pmatrix}
    \gamma_k^1 \\
    \vdots \\
    \gamma_k^q
\end{pmatrix}
\label{eq: a_kl_c_k}
\end{equation}
and
\begin{equation*}
    \alpha_{kl}^{st}=\langle B_k^TB_l,D_sD_t^T\rangle_F \quad \text{and} \quad \gamma_k^s=\langle B_k^T, D_s \rangle_F.
\end{equation*}
We will prefer those latter expressions due to their analogy with the rank $1$ case. Moreover, similarly to the rank $1$ case, the residual may be cheaply evaluated without forming the Kronecker products explicitly. Indeed, similarly to \cref{eq: residual}, we obtain
\begin{align*}
    \|I-\sum_{s=1}^q\sum_{k=1}^r A_kC_s \otimes B_kD_s\|_F^2&=nm-2\sum_{k=1}^r\sum_{s=1}^q \gamma_k^s \delta_k^s +\sum_{k,l=1}^r\sum_{s,t=1}^q \alpha_{kl}^{st}\beta_{kl}^{st} \\
    &=nm-2\sum_{k=1}^r c_k \cdot d_k +\sum_{k,l=1}^r \langle a_{kl}, b_{kl} \rangle_F.
\end{align*}
Thus, apart from the proliferation of indices, the rank $q$ case does not lead to any major additional difficulty. In practice, $\mathcal{A}^T\mathcal{A}$ and $\mathcal{B}^T\mathcal{B}$ are formed one block at a time using \cref{algo: rank_1_approximate_inverse}. The final algorithm is structurally similar to \cref{algo: rank_1_approximate_inverse} and is omitted. It is however important to note that the initial factor matrices $C_s$ must be linearly independent for $\mathcal{B}^T\mathcal{B}$ to be invertible.

The complexity analysis for the Kronecker rank $q$ approximation is more involved but essentially follows the same lines as the Kronecker rank $1$ case. \textbf{If all factor matrices $A_k$ and $B_k$ are dense} and $N$ iterations of the algorithm are necessary, the final cost amounts to $O(Nrq^2(n^3+m^3)+Nr^2q^2(n^2+m^2))$. Although this cost might seem significant at a first glance, we must recall that the total number of iterations $N$ and the rank $q$ are controlled by the user and take small integer values.

Contrary to approximations of the operator, our approximations of the inverse are generally not symmetric, even if the operator is. This problem is reminiscent of sparse approximate inverse techniques \cite{grote1997parallel}. Fortunately, symmetry of the factor matrices can be easily restored by retaining their symmetric part, which experimentally did not seem to have any detrimental effect on the preconditioning quality. More importantly, the algorithm may deliver an exceedingly good data sparse representation of the inverse. Moreover, applying the preconditioning operator 
\begin{equation}
\label{eq: operator_P}
    \mathcal{P}(X)=\sum_{s=1}^q D_sXC_s^T
\end{equation}
only requires computing a few matrix-matrix products, which is generally much cheaper than solving standard Sylvester equations.

Since we are directly approximating the inverse, bounds on the eigenvalues of the preconditioned matrix can be obtained straightforwardly. The theory was already established in the context of sparse approximate inverse preconditioning \cite{grote1997parallel}. Thanks to \cite[Theorem 3.2]{grote1997parallel}, the quality of the clustering of the eigenvalues of the preconditioned matrix is monitored since $\|I-MP\|_F^2$ is evaluated at each iteration and coincides with the stopping criterion of the alternating least squares algorithm. Following the arguments presented in \cite[Theorem 3.1, Corollary 3.1, Theorem 3.2]{grote1997parallel}, it is also possible to state sufficient conditions guaranteeing invertibility of the preconditioning matrix and derive estimates for the iterative condition number of the preconditioned matrix.

\begin{remark}
The minimization problem in the spectral norm was recently considered in \cite{dressler2023kronecker} and could have interesting applications for preconditioning. However, computational methods are still in their infancy and not yet suited for large scale applications.    
\end{remark}

\subsection{Kronecker rank $q$ sparse approximate inverse}
\label{se: sp_rank_q_kinv}
It is well-known that the entries of the inverse of a banded matrix are decaying in magnitude (although non-monotonically) away from the diagonal \cite{demko1984decay, benzi1999bounds}. In the case of block-banded matrices with banded blocks, the inverse features two distinctive decaying patterns: a global decay on the block level as well as a local decay within each individual block \cite{canuto2014on, benzi2015bounds}. Similarly to approximating the inverse of a banded matrix by a banded matrix, the inverse of Kronecker products of banded matrices could also be approximated by Kronecker products of banded matrices. Before explaining how to obtain such an approximation, we must first recall the construction of sparse approximate inverses.

\subsubsection{Sparse approximate inverse techniques}
\label{se: sparse_approximate_inverse}
We begin by recalling some of the basic ideas behind sparse approximate inverse techniques, as they were outlined in \cite{grote1997parallel, chow1998approximate, benzi1999comparative}. Given a sparse matrix $M \in \mathbb{R}^{n \times n}$, the problem consists in finding a sparse approximate inverse of $M$ with a prescribed sparsity pattern. Let $S \subseteq \{(i,j) \colon 1 \leq i,j \leq n \}$ be a set of pairs of indices defining a sparsity pattern and $\mathcal{S}=\{P \in \mathbb{R}^{n \times n}\colon p_{ij}=0, \ (i,j) \notin S\}$ be the associated set of sparse matrices. We then consider the constrained minimization problem
\begin{equation*}
    \min_{P \in \mathcal{S}} \|I-MP\|_F^2
\end{equation*}
where the approximate inverse now satisfies a prescribed sparsity. Noticing that $\|I-MP\|_F^2=\sum_{j=1}^n \|e_j- Mp_j\|_2^2$, each column of $P$ can be computed separately by solving a sequence of independent least squares problems. Since all columns are treated similarly, we restrict the discussion to a single one, denoted $p_j$. Let $\mathcal{J}$ be the set of indices of nonzero entries in $p_j$. Since the multiplication $Mp_j$ only involves the columns of $M$ associated to indices in $\mathcal{J}$, only the submatrix $M(:,\mathcal{J})$ must be retained, thereby drastically reducing the size of the least squares problem. The problem can be further reduced by eliminating the rows of the submatrix $M(:,\mathcal{J})$ that are identically zero (as they will not affect the least squares solution). Denoting $\mathcal{I}$ the set of indices of nonzero rows, the constrained minimization problem turns into a (much smaller) unconstrained problem
\begin{equation*}
    \min_{\hat{p}_j} \|\hat{e}_j-\hat{M}\hat{p}_j\|_2^2
\end{equation*}
where $\hat{M}=M(\mathcal{I}, \mathcal{J})$, $\hat{p}_j=p_j(\mathcal{J})$ and $\hat{e}_j=e_j(\mathcal{I})$. The greater the sparsity of the matrices, the smaller the size of the least squares problem, which is usually solved exactly using a QR factorization. The procedure is then repeated for each column of $P$. Instead of prescribing the sparsity pattern, several authors have proposed adaptive strategies to iteratively augment it until a prescribed tolerance is reached. For simplicity, we will not consider such techniques here and instead refer to the original articles \cite{grote1997parallel, chow1998approximate, benzi1999comparative} for further details. It goes without saying that sparse approximate inverse techniques can only be successful if the inverse can be well approximated by a sparse matrix. Although it might seem as a rather restrictive condition, it is frequently met in applications. In the next section, we will combine low Kronecker rank approximations with sparse approximate inverse techniques. In effect, it will allow us to compute low Kronecker rank approximations of the inverse with sparse factor matrices. 

\subsubsection{Low Kronecker rank sparse approximate inverse}
We first consider again the Kronecker rank 1 approximation. We seek factor matrices $C \in \mathcal{S}_C$ and $D \in \mathcal{S}_D$ where $\mathcal{S}_C$ and $\mathcal{S}_D$ are sets of sparse matrices with prescribed sparsity defined analogously as in \cref{se: sparse_approximate_inverse}. Recalling Equation \cref{eq: transformation} from \cref{se: rank_1_kinv}, we have
\begin{equation*}
    \|I-\sum_{k=1}^r A_kC \otimes B_kD\|_F^2=\|\tilde{I}-\mathcal{B}D\|_F^2=\sum_{j=1}^m \|\tilde{e}_j-\mathcal{B}d_j\|_2^2.
\end{equation*}
We now proceed analogously to \cref{se: sparse_approximate_inverse} and solve a sequence of independent least squares problems for each column of $D$. Let $\mathcal{J}$ be the set of indices corresponding to nonzero entries of $d_j$ and $\mathcal{I}$ be the set of indices for nonzero rows in $\mathcal{B}(:,\mathcal{J})$. We then solve the unconstrained problem
\begin{equation}
\label{eq: red_least_squares}
    \min_{\hat{d}_j} \|\hat{e}_j-\hat{\mathcal{B}}\hat{d}_j\|_2^2
\end{equation}
where $\hat{\mathcal{B}}=\mathcal{B}(\mathcal{I}, \mathcal{J})$, $\hat{d}_j=d_j(\mathcal{J})$ and $\hat{e}_j=\tilde{e}_j(\mathcal{I})$. Contrary to standard sparse approximate inverse techniques, we will not rely on a QR factorization of $\hat{\mathcal{B}}$ but on the normal equations. The solution of the least squares problem in \cref{eq: red_least_squares} is the solution of the linear system $\hat{\mathcal{B}}^T\hat{\mathcal{B}}\hat{d}_j=\hat{\mathcal{B}}^T\hat{e}_j$. Furthermore, we notice that 
\begin{align*}
    &\hat{\mathcal{B}}^T\hat{\mathcal{B}}=\mathcal{B}(\mathcal{I}, \mathcal{J})^T\mathcal{B}(\mathcal{I}, \mathcal{J})=(\mathcal{B}^T\mathcal{B})(\mathcal{J}, \mathcal{J}), \\
    &\hat{\mathcal{B}}^T\hat{e}_j=\mathcal{B}(\mathcal{I}, \mathcal{J})^T\tilde{e}_j(\mathcal{I})=(\mathcal{B}^T\tilde{e}_j)(\mathcal{J}).
\end{align*}
Therefore, the required system matrix and right-hand side vector are simply submatrices of $\mathcal{B}^T\mathcal{B}$ and $\mathcal{B}^T\tilde{I}$, respectively. These quantities are formed only once at each iteration and appropriate submatrices are extracted for computing each column of $D$. This strategy is very advantageous given that forming $\mathcal{B}^T\mathcal{B}$ is rather expensive. The strategy for computing $C$ is again analogous. Overall, computing sparse factors only requires minor adjustments to \cref{algo: rank_1_approximate_inverse}. The case of a Kronecker rank $q$ sparse approximate inverse is not much more difficult. As we have seen in \cref{se: rank_q_kinv}, 
\begin{equation*}
    \|I-\sum_{s=1}^q\sum_{k=1}^r A_kC_s \otimes B_kD_s\|_F^2=\|\tilde{I}-\mathcal{B}D\|_F^2
\end{equation*}
where $D=[D_1; \dots; D_q]$ and $\mathcal{B}$ is defined in \cref{eq: coeff_mat_ls}. We then apply exactly the same strategy as for the Kronecker rank 1 approximation. The only minor difficulty lies in defining suitable sparsity patterns. In our context, we consider powers of $\sum_{k=1}^r A_k$ and $\sum_{k=1}^r B_k$ or variations thereof by adapting well-established strategies for sparse approximate inverses \cite{huckle1999approximate}. Our method inherits many other key properties of sparse approximate inverses, including the possibility to compute columns in parallel.

Apart from obvious storage savings, sparse approximate inverses further speed up the application of the preconditioning operator $\mathcal{P}$ and may even maintain some sparsity in the iterates for sparse input data. This fact was already recognized in \cite{palitta2018numerical} but the construction of such an operator has remained unattended. We formalize the result for the general multiterm equation. Let $\beta_M$ denote the bandwidth of a matrix $M$. For a starting matrix $X_0=0$, the next lemma provides an upper bound on the growth of the bandwidth for the iterates of Bi-CGSTAB.

\begin{lemma}
\label{lem: sparsity_bicgstab}
The Bi-CGSTAB method applied to \cref{eq: operator_M} and preconditioned with \cref{eq: operator_P} with starting matrix $X_0=0$ produces iterates $X_j$ (for a full iteration $j \geq 1$) with bandwidth
\begin{equation*}
    \beta_{X_j} \leq (2j-1)(\beta_{\mathcal{M}}+\beta_{\mathcal{P}})+\beta_{\mathcal{P}}+\beta_E
\end{equation*}
where $\beta_{\mathcal{M}}=\max_{k} \{\beta_{A_k}+\beta_{B_k}\}$ and $\beta_{\mathcal{P}}=\max_{s} \{\beta_{C_s}+\beta_{D_s}\}$.
\end{lemma}
\begin{proof}
The proof is a straightforward adaptation and generalization of \cite[Proposition 2.4 and Proposition 4.1]{palitta2018numerical}.    
\end{proof}

\section{Numerical experiments}
\label{se: experiments}
We now test our preconditioning strategies on a few benchmark problems. All algorithms\footnote{The algorithms and code for reproducing the experiments are freely available at: \url{https://github.com/YannisVoet/Sylvester/tree/main}} are implemented in MATLAB R2023a and run on MacOS with an M1 chip and 32 GB of RAM. The experiments are meant to test our preconditioning techniques for a variety of problems. In the sequel, Kronecker rank $q$ preconditioners given by the nearest Kronecker product and (sparse) Kronecker product approximations of the inverse are referred to as NKP($q$) and KINV($q$), respectively.

\subsection{RC circuit simulation}
Our next example stems from a second order Carleman bilinearization of a nonlinear control system encountered for RC circuit simulations \cite{bai2006projection}. It is a prototypical example of a \emph{Lyapunov-plus-positive} equation and has become over the years a classical benchmark for low-rank solvers \cite{benner2013low, shank2016efficient, kressner2015truncated}. It reads
\begin{equation}
\label{eq: circuit_model}
     AX + XA^T + NXN^T = E
\end{equation}
where $A$, $N$ and $E$ are of size $n=n_0^2+n_0$ and are given by
\begin{equation*}
    A=
    \begin{pmatrix}
        A_1 & A_2 \\
        0 & A_1 \otimes I + I \otimes A_1
    \end{pmatrix},
    \quad
    N=
    \begin{pmatrix}
      0 & 0 \\
      \mathbf{b} \otimes I + I \otimes \mathbf{b} & 0
    \end{pmatrix},
    \quad
    E=-
    \begin{pmatrix}
        \mathbf{b}\\
        0
    \end{pmatrix}
    \begin{pmatrix}
        \mathbf{b}^T & 0
    \end{pmatrix}
\end{equation*}
where $A_1 \in \mathbb{R}^{n_0 \times n_0}, A_2 \in \mathbb{R}^{n_0 \times n_0^2}$ and $\mathbf{b} \in \mathbb{R}^{n_0}$. We refer to \cite{bai2006projection} for the explicit construction of the various blocks. Since we do not exploit any low-rank structure of the solution matrix, we restrict our experiments to small or medium size versions of the equation. Common preconditioning strategies for this problem are based on the Lyapunov part of the equation; i.e. $AX + XA^T$. Although it is usually replaced by a few steps of ADI for computational efficiency, we use it as such in our experiments. For solving \cref{eq: circuit_model}, we set $n_0=30$ ($n=930$) and test our algebraic preconditioners with a restarted GMRES method (with a relative tolerance of $10^{-8}$ and restarted every $50$ iterations). The sparsity pattern for the factor matrices of the approximate inverse is defined based on small powers ($2$ or $4$) of the factor matrices of the operator. The convergence history is shown in \cref{fig: exp4_circuit_n0_30_gmres} for the first $100$ iterations and timings are reported in \cref{tab: timings_circuit}. Interestingly, we notice that the Lyapunov preconditioner and the NKP(2) preconditioner give exactly the same results. As a matter of fact, the two preconditioners are exactly the same, which is a consequence of the orthogonality $\langle A,N \rangle_F=\langle I,N \rangle_F=0$. Indeed, going through the steps of \cref{algo: NKP_preconditioner}, it turns out that $V_Y=V_AW_A$ and $V_Z=V_BW_B$, where $W_A:=R_A^{-1}\tilde{U}\Sigma^{1/2}$ and $W_B:=R_B^{-1}\tilde{V}\Sigma^{1/2}$ have the following form
\begin{equation*}
    \begin{pmatrix}
        * & * & 0 \\
        * & * & 0 \\
        0 & 0 & 1
    \end{pmatrix}
\end{equation*}
where $*$ denotes a nonzero entry. It follows that the coefficient matrices for the NKP(2) preconditioner are merely linear combinations of $A$ and $I$ and do not involve $N$. Thus, $P=AX + XA^T$ is equal to the Lyapunov part, which better explains the success of this preconditioner in \cite{benner2013low}. This observation extends to other settings in \cite{benner2013low, damm2008direct, shank2016efficient}, where the coefficient matrices of the Lyapunov part are ``weakly correlated'' to the additional terms, although non-orthogonal. From a timewise perspective, preconditioners based on solving small-size Lyapunov equations directly are too expensive, despite the small iteration count and problem size. In all our experiments, we have used an in-house implementation of the advanced block splitting strategy of \cite{jonsson2002blocked} coupled with the method in \cite{gardiner1992solution} for solving small-size Sylvester equations. Although replacing it with a few steps of ADI is a natural alternative for this problem, the choice of inner solver is generally not so obvious and such investigations fall outside the scope of this article. Although the KINV preconditioners are much less effective in terms of iteration count, their low application cost makes up the difference and their setup leverages MATLAB's parallel computing capabilities.

\begin{figure}[htbp]
    \centering
    \includegraphics[scale=0.40]{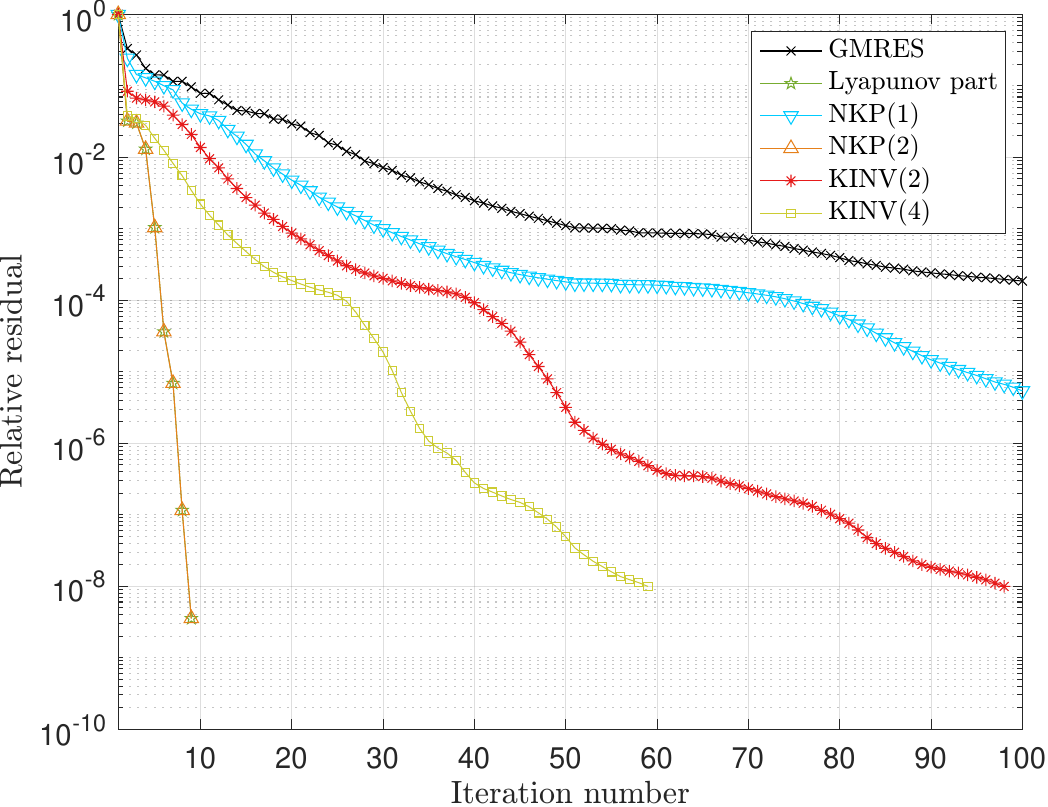}
    \caption{Convergence history for solving \cref{eq: circuit_model} using the (right-preconditioned) GMRES method. The non-preconditioned method converged after $630$ iterations.}
    \label{fig: exp4_circuit_n0_30_gmres}
\end{figure}

\begin{table}[htbp]
    \centering
    \begin{tabular}{|l|l l|}
        \hline
         Preconditioner & Setup & GMRES \\
         \hline
         None & $-$ & 26.0 (630) \\
         Lyapunov & $-$/6.0 & 12.2 (8) \\
         NKP(1) & 0.06/0.02 & 15.0 (203)\\
         NKP(2) & 0.06/5.9 & 12.2 (8)\\
         KINV(2) & 1.4 & 5.8 (97)\\
         KINV(4) & 2.4 & 5.2 (58)\\
         \hline
    \end{tabular}
    \caption{Timing (in seconds). When writing $x/y$, $x$ represents the time for computing the SVD representation of the operator (with \cref{algo: NKP_preconditioner}) and $y$ is the time for computing matrix factorizations (e.g. QZ or LU). The total number of iterations is shown in parenthesis.}
    \label{tab: timings_circuit}
\end{table}

\subsection{Isogeometric analysis}
Our next experiment arises from applications in isogeometric analysis, a tensorized finite element method \cite{hughes2005isogeometric, cottrell2009isogeometric}. It is well-know that discretizations of the Laplacian on simple 2D geometries with separable coefficients lead to a Kronecker rank 2 stiffness matrix \cite{sangalli2016isogeometric}. Unfortunately, this pleasant structure only holds for idealized problems. For nontrivial (single-patch) geometries, the stiffness matrix is nevertheless well approximated by low Kronecker rank matrices; a property at the heart of several fast assembly algorithms \cite{mantzaflaris2017low, scholz2018partial, hofreither2018black}. The same holds true for the mass matrix. In this experiment, we consider a quarter of a plate with a hole commonly used for benchmarking purposes in isogeometric analysis (see e.g. \cite[Figure 16]{hughes2005isogeometric}). For this experiment, we have used GeoPDEs \cite{vazquez2016new}, a MATLAB-Octave software package for isogeometric analysis. The geometry is discretized with cubic splines and $200$ subdivisions in each spatial direction. For this problem, the mass matrix is approximated up to machine precision by a Kronecker rank 10 matrix independently of the discretization parameters, such that
\begin{equation}
\label{eq: mass_operator}
    \mathcal{M}(X)=\sum_{k=1}^{10} B_kXA_k^T. 
\end{equation}
where all the $A_k$ and $B_k$ for $k=1,\dots,10$ are banded matrices of size $n=201$ and $m=403$ with small bandwidth. Although suggested in \cite{scholz2018partial}, few attempts have successfully exploited this structure and instead practitioners often explicitly form the system matrices. Nevertheless, several very efficient preconditioning strategies have been developed \cite{gao2013kronecker, gao2014fast, loli2021easy}. To the best of our knowledge, the preconditioner of Loli et al. \cite{loli2021easy} is the state-of-the-art preconditioner for the isogeometric mass matrix and provides a good comparison for our methods. Although originally formulated for the explicitly assembled mass matrix, the preconditioner may be recast as a linear operator
\begin{equation*}
    \mathcal{P}(X)=S \ast (P_2^{-1}(S \ast X)P_1^{-1})
\end{equation*}
where $P_1,P_2$ are banded matrices, $S$ is a dense low-rank matrix and $\ast$ denotes the Hadamard (elementwise) product. We refer to the original article \cite{loli2021easy} for its explicit construction. We compare the performance of this preconditioner against the algebraic NKP($q$) and KINV($q$) preconditioners for $q=1,2$. Small powers (e.g. $3$ or $4$) of the coefficient matrices define the sparsity pattern of the coefficients for the approximate inverse. Although the mass matrix is symmetric positive definite, not all preconditioners are. We use the Bi-CGSTAB method for all cases to provide a valid comparison. The right-hand side is a rectangular matrix consisting of the identity in its upper block and all zeros in its lower block. We choose $X_0=0$ as starting matrix, set a tolerance of $10^{-8}$ on the relative residual and gap the number of iterations at $100$. Figure \ref{fig: exp2_iga_plate_with_hole_n200_p3_bicgstab} shows the convergence history for the first $50$ iterations. All our preconditioning techniques compare favorably with the preconditioner of Loli et al. in terms of iteration count. The NKP(1) and KINV(1) preconditioners behave similarly while the KINV(2) preconditioner is slightly better than NKP(2). Computing times were all within one second, except in the non-preconditioned case. The KINV(2) preconditioner was the fastest within Bi-CGSTAB but did not outperform the preconditioner of Loli et al. when also considering its setup cost. However, contrary to the preconditioner of Loli et al., the KINV preconditioner maintains a certain level of sparsity throughout the iterations and we could verify that the upper bound of \cref{lem: sparsity_bicgstab} was attained. This becomes a major advantage of our method for large scale problems. Although the NKP(2) preconditioner leads to small iteration counts, it was the most expensive of all. The cost of repeatedly solving standard Sylvester equations explains the difference, which we clearly noticed when comparing the Kronecker rank $1$ and $2$ versions of the preconditioner. The prohibitive cost of repeatedly solving such equations was already stressed in \cite{damm2008direct} and alternative methods should be considered. However, it pertains to standard Sylvester equations and falls outside the scope of this contribution. 

\begin{figure}[htbp]
    \centering
    \includegraphics[scale=0.40]{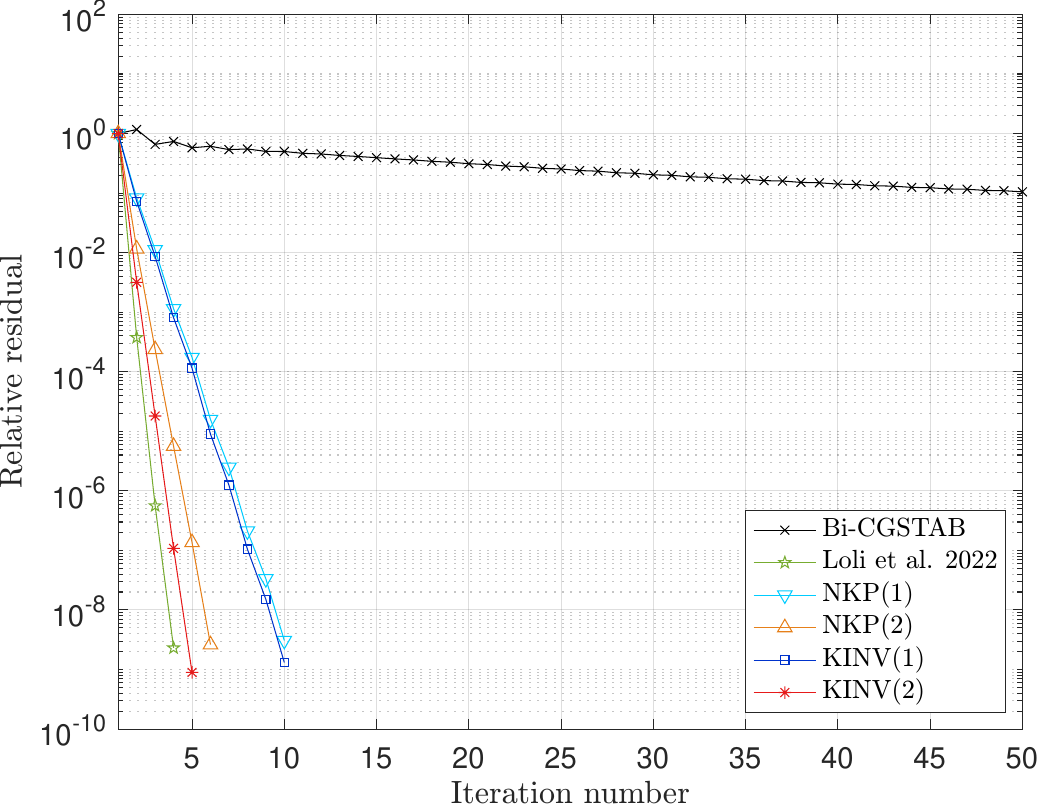}
    \caption{Convergence history for solving \cref{eq: mass_operator} using the (right-preconditioned) Bi-CGSTAB method}
    \label{fig: exp2_iga_plate_with_hole_n200_p3_bicgstab}
\end{figure}

\subsection{Convection-diffusion equation}
We now consider the finite difference discretization of the convection-diffusion equation
\begin{equation*}
    -\epsilon \Delta u + \mathbf{w} \cdot \nabla u = f
\end{equation*}
on the unit square $\Omega=(0,1)^2$ with a convection vector whose components are separable functions (i.e. $w_k(x,y)=\phi_k(x)\psi_k(y)$, $k=1,2$). Under these assumptions, Palitta and Simoncini \cite[Proposition 1]{palitta2016matrix} showed that the discrete solution on a finite difference grid $\{x_i\}_{i=1}^n \times \{y_j\}_{j=1}^n$ satisfies the matrix equation
\begin{equation}
\label{eq: conv_diff_eq}
    TX + XT^T + (\Phi_1 B)X \Psi_1 + \Phi_2 X(B^T\Psi_2) = F
\end{equation}
where $\Phi_k = \diag(\phi_k(x_1),\dots,\phi_k(x_n))$, $\Psi_k=\diag(\psi_k(y_1),\dots,\psi_k(y_n))$ for $k=1,2$ and $B$ and $T$ stem from centered finite difference discretizations of the first and second order derivatives (with diffusion coefficient), respectively, and $F$ accounts for the right-hand side and boundary conditions; see \cite{palitta2016matrix} for the details. Our experiment closely follows Example 4 in \cite{palitta2016matrix}. We set $f=0$,
\begin{equation*}
    \mathbf{w}=
    \begin{pmatrix}
        y(1-(2x+1)^2) \\
        -2(2x+1)(1-y^2)
    \end{pmatrix}
\end{equation*}
and prescribe homogeneous Dirichlet boundary conditions on the entire boundary, except for $y=0$, where
\begin{equation*}
    u(x,0)=
    \begin{cases}
        1 + \tanh(10+20(2x-1)) & 0 \leq x \leq 0.5, \\
        2 & 0.5 < x \leq 1.
    \end{cases}
\end{equation*}
These boundary conditions are built in the right-hand side $F$ following the procedure described in \cite[Section 3]{palitta2016matrix}. Preconditioning \cref{eq: conv_diff_eq} becomes challenging for convection dominated problems (for $\|\mathbf{w}\| \gg \epsilon$), where preconditioners based on the Lyapunov part are ineffective. For this type of problem, the authors constructed a Kronecker rank $2$ preconditioner given by
\begin{equation*}
    \mathcal{P}(X) := (T+\bar{\psi}_1\Phi_1 B)X + X(T+\bar{\phi}_2 \Psi_2B)^T
\end{equation*}
where $\bar{\psi}_1$ and $\bar{\phi}_2$ are the mean of $\{\psi_1(x_i)\}_{i=1}^n$ and $\{\phi_2(y_j)\}_{j=1}^n$, respectively. We solve \cref{eq: conv_diff_eq} for $n=1000$ using GMRES with a relative tolerance of $10^{-6}$ and a maximum number of $200$ iterations. In \cite{palitta2016matrix}, one-sided standard Sylvester equations with the preconditioning operator are solved using with a projection technique based on an extended Krylov subspace method (KPIK) \cite{simoncini2007new}. However, the NKP(2) preconditioning operator is generally two-sided. In order to provide a fair comparison, small-size Sylvester equations are solved in both cases using our in-house implementation of the block recursive splitting algorithm described in \cite{jonsson2002blocked}. Sparse KINV preconditioners are constructed from the sparsity pattern of $(|A|^T|A|)^p$ where $A$ is the sum of all right-sided (or left-sided) coefficient matrices and $p \leq 19$. Their setup was accelerated using MATLAB's parallel computing toolbox. The convergence history is shown in \cref{fig: exp7_convection_diffusion_gmres_rhs0} for the largest and smallest diffusion coefficients. For $\epsilon=1/10$, the NKP(2) preconditioner is nearly as good as the one of Palitta and Simoncini but shows early signs of instabilities for $\epsilon=1/30$. Actually, both operators become increasingly ill-conditioned as $\epsilon \to 0$ and eventually break down. The KINV preconditioners were far more robust and actually improve as $\epsilon \to 0$. Timings for GMRES are reported in \cref{tab: timings_convection_diffusion} for various diffusion coefficients. The timings for the NKP(2) preconditioner and the one of Palitta and Simoncini are rather pessimistic due to our choice of inner solver. It could essentially be replaced with any state-of-the-art method known for this specific problem. Much better timings are reported in \cite[Table 4]{palitta2016matrix} with KPIK and compete with our KINV preconditioner for a problem of comparable size.

\begin{figure}[htbp]
     \centering
     \begin{subfigure}[t]{0.48\textwidth}
    \centering
    \includegraphics[width=\textwidth]{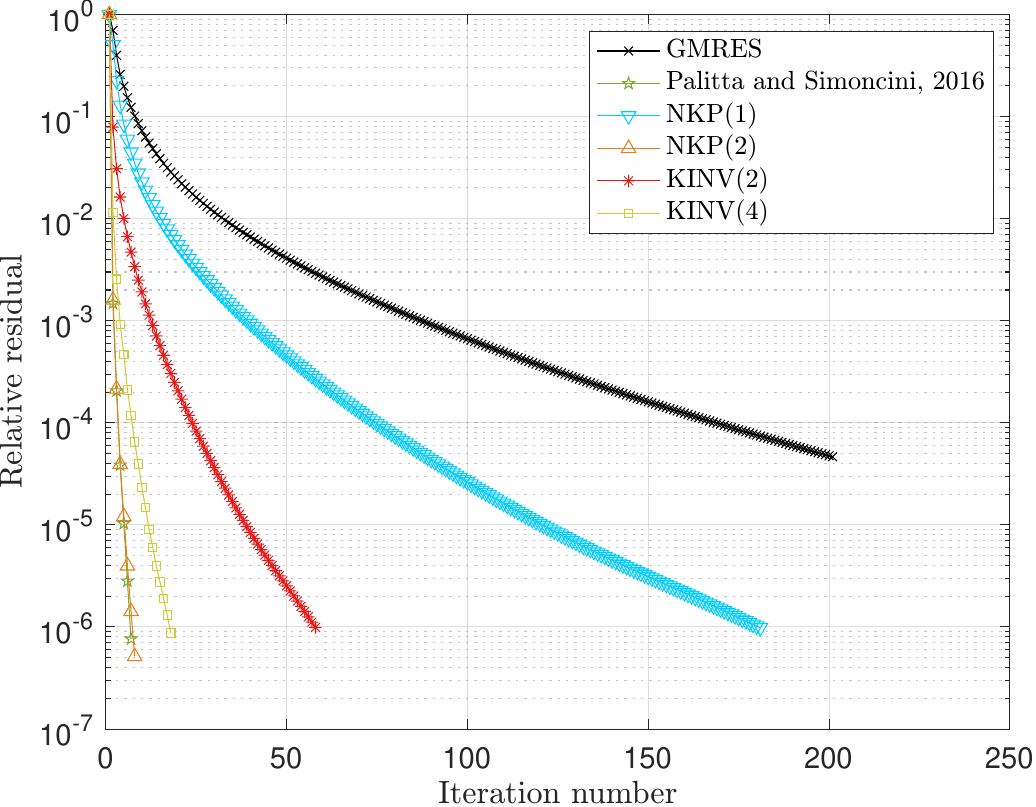}
    \caption{$\epsilon=1/10$}
    \label{fig: exp7_convection_diffusion_gmres_rhs0_eps_0_1}
     \end{subfigure}
     \hfill
     \begin{subfigure}[t]{0.48\textwidth}
    \centering
    \includegraphics[width=\textwidth]{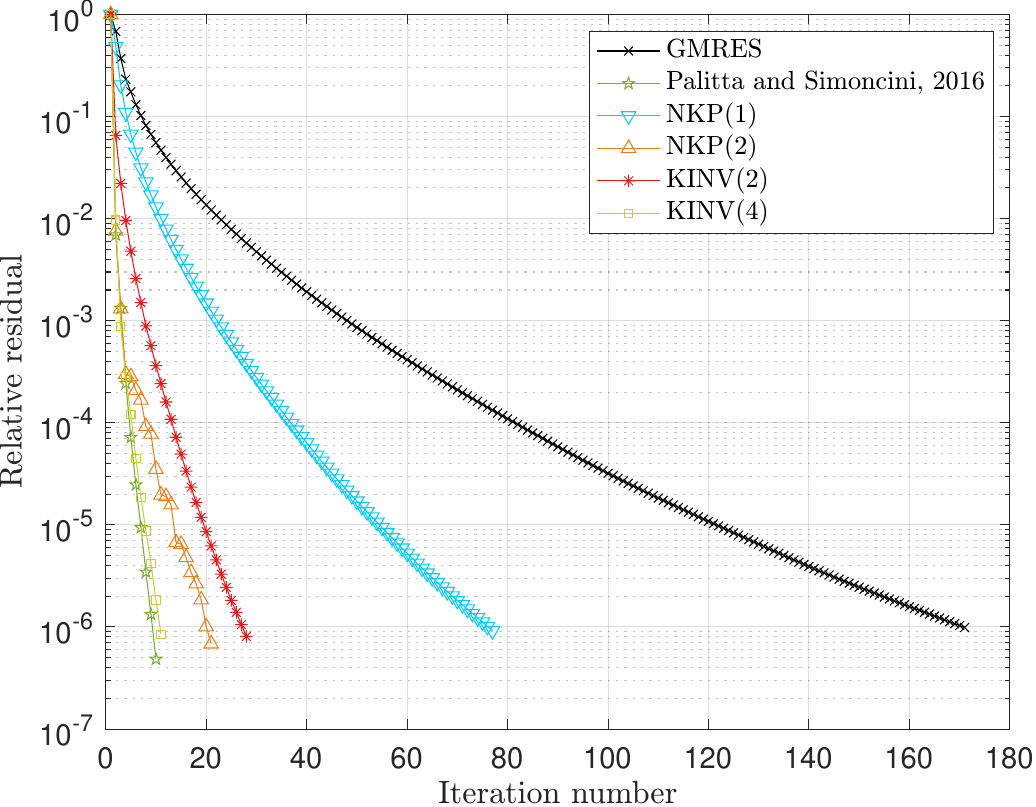}
    \caption{$\epsilon=1/30$}
    \label{fig: exp7_convection_diffusion_gmres_rhs0_eps_0_033}
     \end{subfigure}
     \hfill
    \caption{Convergence history for solving \cref{eq: conv_diff_eq} using the (right-preconditioned) GMRES method}
    \label{fig: exp7_convection_diffusion_gmres_rhs0}
\end{figure}

\begin{table}[htbp]
    \centering
    \begin{tabular}{|l|l l l l|}
        \hline
         Preconditioner & Setup & $\epsilon=1/10$ &  $\epsilon=1/20$ & $\epsilon=1/30$ \\
         \hline
         None & $-$ & 103.9 (200$^*$) & 105.1 (200$^*$) & 73.25 (170) \\
         Palitta and Simoncini & $-$/10.3 & 10.8 (6) & 14.2 (8) & 24.4 (9) \\
         NKP(1) & 0.04/0.01 & 94.9 (180) & 28.9 (104) & 13.8 (76) \\
         NKP(2) & 0.04/8.95 & 13.4 (7) & 20.7 (12) & 51.1 (20) \\
         KINV(2) & 1.04 & 9.38 (57) & 4.41 (35) & 2.95 (27) \\
         KINV(4) & 1.86 & 2.04 (17) & 1.40 (12) & 1.17 (10) \\
         \hline
    \end{tabular}
    \caption{Timing (in seconds). When writing $x/y$, $x$ represents the time for computing the SVD representation of the operator (with \cref{algo: NKP_preconditioner}) and $y$ is the time for computing matrix factorizations (e.g. QZ or LU). The total number of iterations is shown in parenthesis, where $*$ indicates that the method did not converge within the maximum number of iterations.}
    \label{tab: timings_convection_diffusion}
\end{table}

\section{Conclusion}
\label{se: conclusion}
In this article, we have proposed general algebraic parameter-free preconditioning techniques for the iterative solution of generalized multiterm Sylvester matrix equations. Our strategies rely on low Kronecker rank approximations of either the operator or its inverse. While the former requires solving standard Sylvester equations at each iteration, the latter only requires matrix-matrix products and provides an inexpensive alternative. Moreover, we have shown how sparse approximate inverse techniques could be combined with low Kronecker rank approximations, thereby speeding up the application of the preconditioning operator and potentially preserving some sparsity in the iterates. Such approximations may also be valuable for other applications seeking data-sparse representations of the inverse.

Numerical experiments revealed the robustness and effectiveness of our techniques for preconditioning multiterm matrix equations arising in a variety of disciplines, including control systems, isogeometric analysis and finite difference discretizations of convection-dominated problems. For all three experiments, the nearest Kronecker product preconditioner favorably competes with state-of-the-art preconditioners tailored to those specific applications. Nevertheless, the preconditioner must be supplemented with efficient solvers for standard Sylvester equations, which must be assessed on a case by case basis. In contrast, sparse low Kronecker rank approximations of the inverse only require a suitable sparsity pattern and may provide cheap alternatives, despite the often larger iteration count.

\section*{Acknowledgments}
I thank Daniel Kressner for valuable discussions and pointers to relevant literature and Espen Sande for carefully reading through this manuscript.

\end{document}